\documentclass[11pt,a4paper]{article}

\usepackage[ngerman,english]{babel}
\usepackage[utf8]{inputenc}

\usepackage[DIV=12]{typearea}

\pdfoutput=1

\usepackage{stmaryrd}

\usepackage[affil-it]{authblk}

\usepackage{xparse}
\usepackage{amsmath, amsfonts, amssymb, amsthm}
\usepackage{mathtools}
\usepackage{enumitem}

\usepackage{microtype}

\usepackage[noadjust]{cite}

\newcommand{\R}{\mathbb{R}}

\newcommand{\N}{\mathbb{N}}

\newcommand{\B}{\mathcal{B}}

\DeclareDocumentCommand\dd{ o g d() }{
	\IfNoValueTF{#2}{
		\IfNoValueTF{#3}
			{\mathrm{d}\IfNoValueTF{#1}{}{^{#1}}}
			{\mathinner{\mathrm{d}\IfNoValueTF{#1}{}{^{#1}}\argopen(#3\argclose)}}
		}
		{\mathinner{\mathrm{d}\IfNoValueTF{#1}{}{^{#1}}#2} \IfNoValueTF{#3}{}{(#3)}}
	}

\newcommand{\del}{\partial}
\newcommand{\eps}{\varepsilon}
\newcommand{\sgn}{\operatorname{sgn}}

\newcommand{\vcc}{\vcentcolon}

\DeclarePairedDelimiter\abs{\lvert}{\rvert}
\DeclarePairedDelimiter\norm{\Vert}{\rVert}

\theoremstyle{plain}
\newtheorem{theorem}{Theorem}[section]
\newtheorem{lemma}[theorem]{Lemma}
\newtheorem{proposition}[theorem]{Proposition}

\theoremstyle{definition}
\newtheorem{definition}[theorem]{Definition}

\theoremstyle{remark}
\newtheorem{remark}[theorem]{Remark}

\numberwithin{equation}{section}

\title{Uniqueness of measure solutions for multi-component coagulation equations}

 \author{Sebastian Throm \thanks{\texttt{sebastian.throm@umu.se}}}
  
 \affil{\em Ume{\aa} University,  Department of Mathematics and Mathematical Statistics, 901 87 Umeå }

\date{}

\begin{document}
 \maketitle
 
 \begin{abstract}
  We prove uniqueness of measure solutions for a multi-component version of Smoluchowski's coagulation equation. The result is valid for a broad range of coagulation kernels and allows to include a source term. The classical coagulation equation is also covered as a special case.
 \end{abstract}
 
 \section{Introduction}
 
 Smoluchowski's classical equation is a model to describe an infinitely large system of particles which interact by merging upon collision. More precisely, the equation is given by
 \begin{equation}\label{eq:Smol:classic}
  \del_{t}f(t,x)=\frac{1}{2}\int_{0}^{x}K(x-y,y)f(t,x-y)f(t,y)\dd{y}-f(t,x)\int_{0}^{\infty}K(x,y)f(t,y)\dd{y}
 \end{equation}
where $f(t,x)$ denotes the density of clusters of size/volume $x\in (0,\infty)$ and the coagulation kernel $K$ prescribes the rate at which particles merge (see e.g.\@ \cite{LaM04,Lau15}). Recently, a generalisation of \eqref{eq:Smol:classic} attracted interest mainly in the context of athmospheric sciences, where clusters can be formed by different types of particles \cite{OKO13,EKB20}. More precisely, following the notation introduced in \cite{FLN22}, for a natural number $d\geq 1$ we let
\begin{equation*}
 \R_{*}^{d}\vcc=[0,\infty)^{d}\setminus \{(0,\ldots,0)\}
\end{equation*}
and for $x,y\in\R_{*}^{d}$ with $x=(x_1,\ldots,x_d)$ and $y=(y_1,\ldots,y_d)$ we write
\begin{equation*}
 x<y \qquad \text{if} \qquad x_\ell\leq y_\ell \qquad \text{for all } \ell\in\{1,\ldots,d\} \text{ and } x\neq y.
\end{equation*}
Throughout this work we also use the $\ell^1$-norm on $\R_{*}^{d}$, i.e.\@
\begin{equation*}
 \abs{x}=\sum_{k=1}^{d}\abs{x_k}=\sum_{k=1}^{d}x_{k}\qquad \text{for all } x=(x_1,\ldots,x_d)\in \R_{*}^{d}.
\end{equation*}
We thus consider a particle system, where a cluster of size $x=(x_1,\ldots,x_d)\in\R_{*}^{d}$ is formed by $d$ different components with respective volumes $x_1,\ldots,x_d$. The corresponding generalisation of \eqref{eq:Smol:classic} then reads
 \begin{equation}\label{eq:mult:coag}
  \del_{t}f(t,x)=\frac{1}{2}\int_{0<y<x}K(x-y,y)f(t,x-y)f(t,y)\dd{y}-f(t,x)\int_{\R_{*}^{d}}K(x,y)f(t,y)\dd{y}=\vcc Q(f,f) 
 \end{equation}
 where now $K\colon \R_{*}^{d}\times\R_{*}^{d}\to [0,\infty)$ and $f(t,x)$ denotes the density of clusters of size $x\in\R_{*}^{d}$.
 Since the mathematical treatment is very similar, we include additionally a size dependent source term $\zeta\geq 0$ and study the equation
 \begin{equation}\label{eq:mult:coag:source}
  \del_{t}f(t,x)= Q(f,f)(t,x) +\zeta(x).
 \end{equation}
An important property of \eqref{eq:Smol:classic} is the (formal) conservation of the \emph{total mass}, i.e.\@ 
 \begin{equation}\label{eq:mass:conservation:classic}
 \int_{0}^{\infty}xf(t,x)\dd{x}= \int_{0}^{\infty}xf(0,x)\dd{x} \qquad \text{for all } t\geq 0. 
 \end{equation}
 This property appears natural from the underlying particle system, where particles can only merge but mass is neither created nor destroyed. However, for coagulation kernels with a superlinear growth at infinity, a phenomenon known as \emph{gelation} can be observed \cite{EMP02}, where in fact, mass is lost from the system in finite time. This effect is interpreted by the formation of clusters of infinite size due to the rapid growth of the kernel. However, in this work, we will restrict to the case of mass-conserving systems. The corresponding generalisation of the mass conservation property \eqref{eq:mass:conservation:classic} for the multi-component equation \eqref{eq:mult:coag:source} with source reads
 \begin{equation}\label{eq:mass:conservation}
 \int_{\R_{*}^{d}}xf(t,x)\dd{x}=\int_{\R_{*}^{d}}xf(0,x)\dd{x}+t\int_{\R_{*}^{d}}\zeta(x)\dd{x}  \qquad \text{for all } t\geq 0.
 \end{equation}
 A remarkable feature of \eqref{eq:mult:coag}, which has no counterpart in the classical model \eqref{eq:Smol:classic}, is a \emph{localisation} property in the long-time limit which has been recently shown in \cite{FLN22}. In fact, following the notation used there, we let
 \begin{equation*}
  \theta_0=\frac{\int_{\R_{*}^{d}}xf(0,x)\dd{x}}{\int_{\R_{*}^{d}}\abs{x}f(0,x)\dd{x}}.
 \end{equation*}
It has been proven in \cite{FLN22} that there exists a solution $f$ to \eqref{eq:mult:coag} such that, as $t\to\infty$, the mass distribution $\abs{x}f(t,x)$ localises in the direction of $\theta_0$. More precisely, there exists a function $\delta(t)\to 0$ as $t\to\infty$ such that
\begin{equation}\label{eq:localisation}
 \lim_{t\to \infty}\abs[\bigg]{\int_{\{\delta(t)t^{1/(1-\gamma)}\leq \abs{x}\leq t^{1/(1-\gamma)}/\delta(t)\}\cap\{\abs{x/\abs{x}-\theta_0}\leq \delta(t)\}}\abs{x}f(t,x)\dd{x}-\int_{\R_{*}^{d}}\abs{x}f(0,x)\dd{x}}=0.
\end{equation}
Note that the parameter $\gamma$ arises in the bound on the kernel $K$ while the precise assumptions can be found in \cite[Theorem~1.1]{FLN22}. 
 
The goal of this work is to establish uniqueness of solutions to \eqref{eq:mult:coag:source} which will thus especially prove that the property \eqref{eq:localisation} is a universal feature of this model (at least in the common parameter regime where both statements hold). Note that our result allows also for $\zeta\equiv 0$. In particular, we will also consider measure solutions. This approach has the advantage that it automatically includes the discrete version of \eqref{eq:mult:coag:source} as a particular case by restricting to Dirac measures located at integer values.

Concerning the classical coagulation equation \eqref{eq:Smol:classic}, there are various results in the literature which establish uniqueness of solutions in suitable weighted $L^1$ spaces for different classes of kernels, e.g.\@ \cite{Ste90a,LaM04,EMR05}. However, for measure solutions there are only very few results available. In \cite{Nor99}, existence and uniqueness of measure solutions to \eqref{eq:Smol:classic} is shown relying on the contraction mapping theorem together with a monotonicity argument. The corresponding result includes a broad range of coagulation kernels but, as in the previously mentioned works, requires the boundedness of two different moments of the initial datum. Conversely, in \cite{FoL06a} uniqueness of measure solutions to \eqref{eq:Smol:classic} has been established requiring only the boundedness of one moment for the initial datum. However, the result is restricted to a smaller class of coagulation kernels, and particularly, the important case of Smoluchowski's kernel for Brownian particles, i.e.\@
\begin{equation*}
 K(x,y)=(x^{1/3}+y^{1/3})(x^{-1/3}+y^{-1/3})
\end{equation*}
is not included. In \cite{MeP04} uniqueness of measure solutions has also been established under mild assumptions on the initial datum for the three \emph{solvable kernels} $K(x,y)=2$, $K(x,y)=x+y$ and $K(x,y)=xy$ exploiting explicit formulas for the Laplace transformed equation.

In this work we aim at extending the uniqueness of measure solutions to the multi-component equation \eqref{eq:mult:coag:source} where, to our knowledge, no uniqueness results are so far available. The result we prove here is mainly a generalisation of the uniqueness part in \cite{Nor99} to the multi-component equation \eqref{eq:mult:coag:source} for measures including also a source term. However, our proof is different from the one in \cite{Nor99}. In fact, we combine the approach to prove uniqueness in the $L^1$ setting for the one-component equation with methods developed in \cite{LuM12} for measure solutions for the Boltzmann model.

 \section{Notation, definitions and main results}
 
 \subsection{Notation}
 
 Adapting the notation from \cite{LuM12}, we denote by $\B(\R_*^d)$ the set of Borel measures on $\R_*^d$. Moreover, for $\alpha,\beta\in\R$ we define the weight functions $\omega_{\alpha,\beta}\colon \R_{*}^{d}\to [0,\infty)$ as well as $\overline{\omega}_{\alpha,\beta}\colon (0,\infty)\to [0,\infty)$ through
 \begin{equation}\label{eq:weight}
  \omega_{\alpha,\beta}(x)\vcc=\overline{\omega}_{\alpha,\beta}(\abs{x})\vcc=\begin{cases}
                                \abs{x}^{\alpha} & \text{if } \abs{x}\leq 1\\
                                \abs{x}^{\beta} & \text{if } \abs{x}> 1.
                               \end{cases}
\end{equation}
We then define
\begin{equation*}
 \B_{\alpha,\beta}(\R_*^d)\vcc=\Bigl\{\mu\in\B(\R_{*}^{d})\;|\; \norm{\mu}_{\alpha,\beta}\vcc=\int_{\R_{*}^{d}}\omega_{\alpha,\beta}(x)\dd{\abs{\mu}(x)}<\infty\Bigr\}
\end{equation*}
where $\abs{\mu}$ is the total variation of $\mu$ and by $\B_{\alpha,\beta}^{+}$ we denote the corresponding cone of non-negative measures. Furthermore, we define
\begin{multline*}
 L_{-\alpha,-\beta}^{\infty}(\R_{*}^{d})=\Bigl\{f\in L^{\infty}(\R_{*}^{d})\; | \;\\*
\norm{f}_{L^{\infty}_{-\alpha,-\beta}}\vcc= \sup_{x\in \R_{*}^{d}}\abs{f(x)}\omega_{-\alpha,-\beta}(x)=\sup_{x\in \R_{*}^{d}}\abs{f(x)}(\omega_{\alpha,\beta}(x))^{-1}<\infty\Bigr\}.
\end{multline*}
By duality, it holds
\begin{equation}\label{eq:duality}
 \norm{\mu}_{\alpha,\beta}=\sup\biggl\{\abs[\bigg]{\int_{\R_{*}^{d}}\varphi(x)\dd{\mu(x)}}\;\Big|\;\varphi\in C_{c}^{k}(\R_{*}^{d}) \text{ such that }\norm{\varphi}_{L_{-\alpha,-\beta}^{\infty}}\leq 1\biggr\} \quad \text{for all } k\in\N\cup \{0\}.
\end{equation}
We introduce moreover the following vector space of test functions which is convenient for working with \eqref{eq:mult:coag:source}:
\begin{equation}\label{eq:test:functions}
 \mathcal{T}\vcc=\bigl\{\varphi\in C_{b}^{0,1}(\R_{*}^{d})\; |\; \exists \eps>0: \abs{x+y}<\eps \Rightarrow \varphi(x+y)-\varphi(x)-\varphi(y)=0\bigr\}
\end{equation}
\begin{remark}
 This means that $\mathcal{T}$ is given by the set of bounded Lipschitz functions on $\R_{*}^{d}$ which are linear in a small neighbourhood of zero. In particular $C_{c}^{\infty}(\R_{*}^{d})\subset \mathcal{T}$.
\end{remark}

\subsection{Definitions}

Concerning the coagulation kernel $K\colon \R_{*}^{d}\times \R_{*}^{d}\to [0,\infty)$, we assume that $K$ is continuous and symmetric, i.e.\@ $K(x,y)=K(y,x)$ for all $x,y\in\R_{*}^{d}$ and there exist $\theta_1,\theta_2\in\R$ such that
\begin{equation}\label{eq:est:kernel:2}
  K(x,y)\leq c_{u}\begin{cases}
                  \abs{x}^{-\theta_1}\abs{y}^{\theta_2} & \text{if } \abs{x}\leq \abs{y}\\
                  \abs{x}^{\theta_2}\abs{y}^{-\theta_1} & \text{if } \abs{y}\leq \abs{x},
                 \end{cases}
                 \quad
                 -\theta_1\leq \theta_2, \quad \theta_2<1\quad \text{and}\quad \gamma\vcc=-\theta_1+\theta_2<1.
\end{equation}
\begin{remark}
Obviously, we also have $K(x,y)\leq c_{u}(\abs{x}^{-\theta_1}\abs{y}^{\theta_2}+\abs{x}^{\theta_2}\abs{y}^{-\theta_1})$ which is a more common form found in the mathematical literature. On the other hand, a kernel satisfying the latter estimate also satisfies \eqref{eq:est:kernel:2} upon changing the constant to $2c_u$. We also note that \eqref{eq:est:kernel:2} contains in particular the class of coagulation kernels considered in \cite{FLN22}. 
\end{remark}

In analogy to the Boltzmann equation in \cite{LuM12}, we define in the following a weak and a strong concept of measure solutions to \eqref{eq:mult:coag:source}. First, we define the weak form of $Q$ for measures through
\begin{equation}\label{eq:def:coag:op:weak}
 \langle Q(\mu,\nu),\varphi\rangle\vcc=\frac{1}{2}\int_{(\R_{*}^{d})^2}K(x,y)[\varphi(x+y)-\varphi(x)-\varphi(y)]\dd{\mu(x)}\dd{\nu(y)}.
\end{equation}
We will use the following notion of weak solutions.
\begin{definition}[weak solution]\label{Def:weak:sol}
 Let $K$ satisfy \eqref{eq:est:kernel:2} and let $f_0\in \B^{+}_{-\theta_1,1}(\R_{*}^{d})$ as well as $\zeta\in \B^{+}_{-\theta_1,1}(\R_{*}^{d})$.
 We denote $\{f_t\}_{t\geq 0}\subset \B^{+}_{\min\{1-\theta_1,1\},1}$ a weak solution with initial datum $f_0$ if 
 \begin{enumerate}
  \item[$\diamond$] $\sup_{t\geq 0}\norm{f_t}_{\min\{1-\theta_1,1\},1}<\infty$
  \item[$\diamond$] $\int_{\R_{*}^{d}}x\dd{f_t(x)}=\int_{\R_{*}^{d}}x\dd{f_0(x)}+t\int_{\R_{*}^{d}}x\dd{\zeta(x)}$ for all $t>0$ (\emph{mass conservation})
  \item[$\diamond$] $t\mapsto \langle Q(f_t,f_t),\varphi\rangle\in C([0,\infty))$ for all $\varphi\in\mathcal{T}$
  \item[$\diamond$] $\int_{\R_{*}^{d}}\varphi(x)\dd{f_t(x)}=\int_{\R_{*}^{d}}\varphi(x)\dd{f_0(x)}+\int_{0}^{t}\langle Q(f_s,f_s),\varphi\rangle\dd{s}+t\int_{\R_{*}^{d}}\varphi(x)\dd{\zeta(x)}$ for all $t\geq 0$ and all $\varphi\in \mathcal{T}$.
 \end{enumerate}
\end{definition}

\begin{remark}\label{Rem:weak:sol}
 Due to the continuity of $t\mapsto \langle Q(f_t,f_t),\varphi\rangle$ each weak solution of \eqref{eq:mult:coag:source} satisfies $\frac{\dd}{\dd{t}}\int_{\R_{*}^{d}}\varphi(x)\dd{f_t(x)}=\langle Q(f_t,f_t),\varphi\rangle+\int_{\R_{*}^{d}}\varphi(x)\dd{\zeta(x)}$ for each $\varphi\in \mathcal{T}$.
\end{remark}

The strong form of the gain and loss term of $Q$ are given through Riesz's representation Theorem via
\begin{equation}\label{eq:def:coag:op:strong}
 \begin{split}
  \int_{\R_{*}^{d}}\varphi(x)\dd{Q^{+}(\mu,\nu)(x)}&\vcc=\int_{(\R_{*}^{d})^2}\varphi(x+y)K(x,y)\dd{\mu(x)}\dd{\nu(y)}\\
  \int_{\R_{*}^{d}}\varphi(x)\dd{Q^{-}(\mu,\nu)(x)}&\vcc=\int_{(\R_{*}^{d})^2}\varphi(x)K(x,y)\dd{\mu(x)}\dd{\nu(y)}\\
  Q(\mu,\nu)&\vcc= Q^{+}(\mu,\nu)-Q^{-}(\mu,\nu)
 \end{split}
\end{equation}
for $\varphi\in L_{-\alpha,-\beta}^{\infty}(\R_{*}^{d})\cap C(\R_{*}^{d})$ with $\alpha,\beta\geq 0$.

\begin{remark}
 It will be proven in Lemma~\ref{Lem:cont:coag:op:gen} that $Q\colon \B_{-\theta_1,\beta+\theta_2}(\R_{*}^{d})\times\B_{-\theta_1,\beta+\theta_2}(\R_{*}^{d}) \to \B_{\alpha,\beta}(\R_{*}^{d})$ for $\alpha,\beta\geq 0$. In particular, for strong solutions as given in Definition~\ref{Def:strong:sol} the coagulation operator $Q$ defined via \eqref{eq:def:coag:op:strong} coincides with \eqref{eq:def:coag:op:weak}. By abuse of notation, we thus denote both operators by $Q$.
\end{remark}
We will then use the following notion of strong solution.
\begin{definition}[strong solution]\label{Def:strong:sol}
 Let $K$ satisfy \eqref{eq:est:kernel:2} and let $f_0\in \B^{+}_{-\theta_1,1}(\R_{*}^{d})$ as well as $\zeta\in \B^{+}_{-\theta_1,1}(\R_{*}^{d})$. We denote $\{f_t\}_{t\geq 0}\subset \B_{-\theta_1,1}^{+}(\R_{*}^{d})$ a strong solution with initial datum $f_0$ if 
 \begin{enumerate}
  \item[$\diamond$] $\sup_{t\geq 0}\norm{f_t}_{\min\{1-\theta_1,1\},1}<\infty$
  \item[$\diamond$] $\int_{\R_{*}^{d}}x\dd{f_t(x)}=\int_{\R_{*}^{d}}x\dd{f_0(x)}+t\int_{\R_{*}^{d}}x\dd{\zeta(x)}$ for all $t>0$ (\emph{mass conservation})
  \item[$\diamond$] $t\mapsto f_t\in C([0,\infty),\B_{-\theta_1,1}(\R_{*}^{d}))\cap C^{1}([0,\infty),\B_{\max\{0,-\theta_1\},0}(\R_{*}^{d}))$ satisfies
  \begin{equation*}
   \frac{\dd}{\dd{t}}f_t=Q(f_t,f_t)+\zeta\qquad \text{for all } t\in[0,\infty).
  \end{equation*}
 \end{enumerate}
\end{definition}

\subsection{Main result and outline}

\begin{theorem}\label{Thm:unique}
 Let $K$ satisfy \eqref{eq:est:kernel:2} and let $f_{0},\zeta\in \B_{-2\theta_1,2\theta_2}^{+}(\R_{*}^{d})\cap \B^{+}_{-\theta_1,1}(\R_{*}^{d})$. There exists a unique weak solution to~\eqref{eq:mult:coag:source} with initial condition $f_0$.
\end{theorem}

\begin{remark}
The existence of weak measure solutions to \eqref{eq:mult:coag:source} can be proven by compactness arguments analogously to the one-component equation \eqref{eq:Smol:classic} as e.g.\@ in \cite{EsM06}. The approach has been outlined in \cite{FLN22} for $\zeta=0$ using a slightly different notion of weak solution and considering a smaller class of coagulation kernels but the method is well established. We will therefore restrict to the proof of the uniqueness part. 
\end{remark}

\begin{remark}
 Theorem~\ref{Thm:unique} can be directly extended to time dependent source terms. However, to avoid technicalities we restrict to stationary source terms.
\end{remark}

The remaining part of this work is organised as follows: In Section~\ref{Sec:weight} we will collect several estimates on the weight function $\omega_{\alpha,\beta}$ which will be used in Section~\ref{Sec:moments} to derive suitable moment estimates on weak solutions of \eqref{eq:mult:coag:source}. The proof of our main statement, Theorem~\ref{Thm:unique}, is contained in Section~\ref{Sec:uniqueness}. The strategy is similar to the one established in \cite{LuM12} for the Boltzmann equation, i.e.\@ the key step is to show first that weak solutions are actually strong solutions. For the latter, we can then prove uniqueness relying on arguments on the sign decomposition for measures. This last step exploits the specific form of \eqref{eq:def:coag:op:weak} which has also been central in corresponding results on the one-component equation in $L^1$ (e.g.\@ \cite{EMR05,Lau15}). However, the proof given here extends both to the multi-component equation and particularly also to measure solutions.

\section{Estimates on the weight function}\label{Sec:weight}

We collect in this section several estimates on the weight function $\omega_{\alpha,\beta}$. As it is well-known, due to the weak form of the coagulation operator, i.e.\@ \eqref{eq:def:coag:op:weak}, a fundamental property will be \emph{subadditivity}. A function $\omega\colon \R_{*}^{d}\to [0,\infty)$ is denoted \emph{subadditive} if $\omega(x+y)\leq \omega(x)+\omega(y)$ for all $x,y\in\R_{*}^{d}$. Following \cite{Nor99}, we denote moreover $\overline{\omega}\colon (0,\infty)\to [0,\infty)$ \emph{sublinear} if $\overline{\omega}(\lambda x)\leq \lambda \overline{\omega}(x)$ for all $\lambda\geq 1$ and all $x\in(0,\infty)$ and recall the basic fact that sublinear functions are subadditive. 

\begin{lemma}\label{Lem:weight:subadditive}
 For $\alpha,\beta\leq 1$ the function $\overline{\omega}_{\alpha,\beta}$ defined in \eqref{eq:weight} is sublinear and hence $\omega_{\alpha,\beta}(\cdot)=\overline{\omega}_{\alpha,\beta}(\abs{\cdot})$ is subadditive, i.e.\@ 
 \begin{equation}\label{eq:weight:subadditive}
  \omega_{\alpha,\beta}(x+y)\leq \omega_{\alpha,\beta}(x)+\omega_{\alpha,\beta}(y) \qquad \text{for all } x,y\in\R_{*}^{d}.
 \end{equation}
 Moreover, for each $n>1$, the function $\varphi_{n}\colon \R_{*}^{d}\to [0,\infty)$ given by $\varphi_{n}(x)\vcc=\min\{\omega_{\alpha,\beta}(x),n\}$ is subadditive, bounded and satisfies $\varphi_{n}(x)\leq \omega_{\alpha,\beta}(x)$ as well as $\varphi_{n}(x)\to \omega_{\alpha,\beta}(x)$ as $n\to\infty$.
\end{lemma}

\begin{proof}
 We prove that $\overline{\omega}_{\alpha,\beta}$ is sublinear. Let $r\in (0,\infty)$ and $\lambda\geq 1$. If $\lambda r\leq 1$, we have in particular $r\leq 1$ and thus
 \begin{equation*}
  \overline{\omega}_{\alpha,\beta}(\lambda r)=(\lambda r)^{\alpha}\leq \lambda r^{\alpha}=\lambda \overline{\omega}_{\alpha,\beta}(r).
 \end{equation*}
 If $\lambda r>1$ and $r>1$, we get similarly
  \begin{equation*}
  \overline{\omega}_{\alpha,\beta}(\lambda r)=(\lambda r)^{\beta}\leq \lambda r^{\beta}=\lambda \overline{\omega}_{\alpha,\beta}(r).
 \end{equation*}
 On the other hand, if $\lambda r>1$ but $r\leq 1$ we have
 \begin{equation*}
  \overline{\omega}_{\alpha,\beta}(\lambda r)=(\lambda r)^{\beta}\leq \lambda r\leq \lambda r^{\alpha}=\lambda \overline{\omega}_{\alpha,\beta}(r).
 \end{equation*}
 Thus, $\overline{\omega}_{\alpha,\beta}$ is sublinear which directly yields the subadditivity of $\omega_{\alpha,\beta}$.
 
 The subadditivity of $\varphi_n$ then directly follows noting that
 \begin{multline*}
  \varphi_{n}(x+y)=\min\{\omega_{\alpha,\beta}(x+y),n\}\leq \min\{\omega_{\alpha,\beta}(x)+\omega_{\alpha,\beta}(y),n\} \\*
  \leq \min\{\omega_{\alpha,\beta}(x),n\} +\min\{\omega_{\alpha,\beta}(y),n\} =\varphi_{n}(x)+\varphi_{n}(y).
 \end{multline*}
 The bound $\varphi_{n}(x)\leq \omega_{\alpha,\beta}(x)$ and $\varphi_{n}(x)\to \omega_{\alpha,\beta}(x)$ as $n\to\infty$ are immediate consequences of the definition.
\end{proof}

The next lemma shows that we can approximate weights which are generated by sublinear functions by bounded functions which are linear near zero (see \eqref{eq:test:functions}). A similar construction has also been used in \cite{Nor99}.

\begin{lemma}\label{Lem:weight:reg:subadditive}
Let $\overline{\omega}\colon (0,\infty)\to [0,\infty)$ be sublinear and let $\eps\in(0,1)$ and $R>1$. Then the function $\varphi_{\eps,R}\colon\R_{*}^{d}\to [0,\infty)$ given by
\begin{equation*}
 \varphi_{\eps,R}(x)\vcc=\overline{\varphi}_{\eps,R}(\abs{x})\vcc=\begin{cases}
                  \eps^{-1}\overline{\omega}(\eps)\abs{x} & \text{if } \abs{x}<\eps\\
                  \min\bigl\{\overline{\omega}(\abs{x}),R\bigr\} &\text{if } \abs{x}\geq \eps
                 \end{cases}
\end{equation*}
is subadditive, $\varphi_{\eps,R}(x)\leq \overline{\omega}(\abs{x})$ and $\varphi_{\eps,R}\to \overline{\omega}(\abs{\cdot})$ as $\eps\to 0$ and $R\to \infty$.
\end{lemma}

\begin{proof}
 To prove subadditivity it suffices again to prove that $\overline{\varphi}_{\eps,R}\colon (0,\infty)\to [0,\infty)$ is sublinear. Let $r\in (0,\infty)$ and $\lambda \geq 1$.
 \begin{itemize}
  \item If $\lambda r <\eps$ it holds in particular $r<\eps$ and we have $\overline{\varphi}_{\eps,R}(\lambda r)=\lambda \eps^{-1}\overline{\omega}(\eps)r=\lambda \overline{\varphi}_{\eps,R}(r)$.
  \item If $\lambda r\geq \eps$ and
  \begin{itemize}
   \item if $r<\eps$ we have $\overline{\varphi}_{\eps,R}(\lambda r)=\min\{\overline{\omega}(\lambda r),R\}\leq \overline{\omega}\bigl(\frac{\lambda r}{\eps}\eps\bigr)\leq \lambda \eps^{-1}\overline{\omega}(\eps)r=\lambda \overline{\varphi}_{\eps,R}(r),$
   \item if $r\geq \eps$ we have $\overline{\varphi}_{\eps,R}(\lambda r)=\min\{\overline{\omega}(\lambda r),R\}\leq \min\{\lambda \overline{\omega}(r),R\}\leq \lambda \overline{\varphi}_{\eps,R}(r)$.
  \end{itemize}
 \end{itemize}
 Summarising, we see that $\overline{\varphi}_{\eps,R}$ is sublinear and thus $\varphi_{\eps,R}$ is subadditive. 
 
 Moreover, if $\abs{x}\geq \eps$, we immediately have $\varphi_{\eps,R}(x)=\min\{\overline{\omega}(\abs{x}),R\}\leq \overline{\omega}(\abs{x})$. On the other hand, if $\abs{x}<\eps$, we get by sublinearity that $\varphi_{\eps,R}(x)=\frac{\abs{x}}{\eps}\overline{\omega}(\frac{\eps}{\abs{x}}\abs{x})\leq \overline{\omega}(\abs{x})$.
\end{proof}

If $\alpha>1$ or $\beta>1$, the function $\omega_{\alpha,\beta}$ is no longer subadditive but $\omega_{\alpha,\beta}(x+y)$ can still be bounded by $\omega_{\alpha,\beta}(x)+\omega_{\alpha,\beta}(y)$ up to a correcting factor.

\begin{lemma}\label{Lem:weight:1}
 For $\alpha,\beta\geq 0$ we have
 \begin{equation*}
  \omega_{\alpha,\beta}(x+y)\leq \max\{2^{\alpha},2^{\beta}\}(\omega_{\alpha,\beta}(x)+\omega_{\alpha,\beta}(y)) \qquad \text{for all } x,y\in\R_{*}^{d}.
 \end{equation*}
\end{lemma}

\begin{proof}
 If $\abs{x+y}\leq 1$ we have in particular $\abs{x},\abs{y}\leq 1$ and thus
 \begin{equation*}
  \omega_{\alpha,\beta}(x+y)=\abs{x+y}^{\alpha}=(\abs{x}+\abs{y})^{\alpha}\leq 2^{\alpha}\max\{\abs{x},\abs{y}\}^{\alpha}\leq 2^{\alpha}(\omega_{\alpha,\beta}(x)+\omega_{\alpha,\beta}(y)).
 \end{equation*}
 If $\abs{x+y}>1$ it holds $\max\{\abs{x},\abs{y}\}\geq 1/2$ and we get
 \begin{equation*}
  \omega_{\alpha,\beta}(x+y)=(\abs{x}+\abs{y})^{\beta}\leq 2^{\beta}\max\{\abs{x},\abs{y}\}^{\beta}.
 \end{equation*}
 If $\max\{\abs{x},\abs{y}\}\geq 1$ we have $\max\{\abs{x},\abs{y}\}^{\beta}\leq (\omega_{\alpha,\beta}(x)+\omega_{\alpha,\beta}(y))$. On the other hand, if $\max\{\abs{x},\abs{y}\}< 1$ we get together with $\max\{\abs{x},\abs{y}\}\geq 1/2$ that 
 \begin{equation*}
 \max\{\abs{x},\abs{y}\}^{\beta}\leq \max\{1,2^{\alpha-\beta}\}\max\{\abs{x},\abs{y}\}^{\alpha}\leq \max\{1,2^{\alpha-\beta}\}(\omega_{\alpha,\beta}(x)+\omega_{\alpha,\beta}(y)). 
 \end{equation*}
 Summarising, we obtain the claim.
\end{proof}

\section{Moment estimates}\label{Sec:moments}

In this section we collect several moment estimates for weak solutions on which we will rely on in the remainder of this work. The first lemma states that for any weak solution all sublinear moments are bounded provided the initial datum and the source have this property.

\begin{lemma}\label{Lem:small:moments:1}
 Let $\alpha,\beta\leq 1$. Each weak solution $f_t$ of \eqref{eq:mult:coag:source} with initial datum $f_0$ and source $\zeta$ such that $f_0,\zeta\in \B_{\alpha,\beta}^{+}\cap\B_{-\theta_1,1}^{+}$ satisfies the moment estimate
 \begin{equation*}
  \int_{\R_{*}^{d}}\omega_{\alpha,\beta}(x)\dd{f_{t}(x)}\leq \int_{\R_{*}^{d}}\omega_{\alpha,\beta}(x)\dd{f_{0}(x)}+t\int_{\R_{*}^{d}}\omega_{\alpha,\beta}(x)\dd{\zeta(x)} \qquad \text{for all } t\geq 0.
 \end{equation*}
 In particular, for each $a\in[-\theta_1,1]$, each weak solution $f_t$ satisfies $\int_{\R_{*}^{d}}\abs{x}^{a}\dd{f_{t}(x)}\leq \norm{f_0}_{-\theta_1,1}+t\norm{\zeta}_{-\theta_1,1}$ for all $t\geq 0$.
\end{lemma}

\begin{proof}
 Fix $\eps\in(0,1)$ and $R>1$ and define $\varphi_{\eps,R}$ as in Lemma~\ref{Lem:weight:reg:subadditive} (with $\overline{\omega}=\overline{\omega}_{\alpha,\beta}$) which is subadditive. From the definition of weak solutions, we obtain thus directly that
  \begin{equation*}
  \int_{\R_{*}^{d}}\varphi_{\eps,R}(x)\dd{f_{t}(x)}\leq \int_{\R_{*}^{d}}\varphi_{\eps,R}(x)\dd{f_{0}(x)}+ t\int_{\R_{*}^{d}}\varphi_{\eps,R}(x)\dd{\zeta(x)} \qquad \text{for all } t\geq 0.
 \end{equation*}
 Since $\varphi_{\eps,R}(x)\leq \omega_{\alpha,\beta}(x)$ we can pass to the limit $\eps\to 0$ and $R\to \infty$ which finishes the proof.
\end{proof}

The next lemma gives a slightly improved moment estimate for negative moments which will be used later to estimate the measure of small regions around zero.

\begin{lemma}\label{Lem:small:moments:2}
 Let $\alpha>0$ and $f_0,\zeta\in \B^{+}_{-\theta_1,1}$ such that $\int_{\R_{*}^{d}}\abs{x}^{-\alpha}\dd{f_0(x)}<\infty$ and $\int_{\R_{*}^{d}}\abs{x}^{-\alpha}\dd{\zeta(x)}<\infty$. There exists a convex function $\Phi\in C^{\infty}([0,\infty),[0,\infty))$ satisfying $\Phi(0)=\Phi'(0)=0$ as well as $\Phi'>0$ on $(0,\infty)$ and $\lim_{r\to\infty}\Phi(r)/r=\infty$ such that each weak solution $f_t$ of \eqref{eq:mult:coag:source} satisfies
 \begin{equation}\label{eq:gen:mom}
  \int_{\R_{*}^{d}}\Phi(\abs{x}^{-\alpha})\dd{f_{t}(x)}\leq \int_{\R_{*}^{d}}\Phi(\abs{x}^{-\alpha})\dd{f_{0}(x)}+t\int_{\R_{*}^{d}}\Phi(\abs{x}^{-\alpha})\dd{\zeta(x)} \qquad \text{for all }t\geq 0.
 \end{equation}
 Note that the properties of $\Phi$ imply in particular that $r\mapsto \Phi(r)/r$ is non-decreasing.
\end{lemma}

\begin{proof}
 By the assumption $\int_{\R_{*}^{d}}\abs{x}^{-\alpha}\dd{f_0(x)}<\infty$  and $\int_{\R_{*}^{d}}\abs{x}^{-\alpha}\dd{\zeta(x)}<\infty$ we deduce that $\lim_{R\to\infty} \int_{\{x\in\R_{*}^{d}\colon\abs{x}^{-\alpha}>R\}}\abs{x}^{-\alpha}\dd{f_0(x)}=0$ and $\lim_{R\to\infty} \int_{\{x\in\R_{*}^{d}\colon\abs{x}^{-\alpha}>R\}}\abs{x}^{-\alpha}\dd{\zeta(x)}=0$. According to \cite[Theorem~8]{Lau15}, there exists a function $\Phi$ as specified in the statement such that 
 \begin{equation}\label{eq:small:moment:2:1}
  \int_{\R_{*}^{d}}\Phi(\abs{x}^{-\alpha})\dd{f_{0}(x)}<\infty \quad \text{and}\quad \int_{\R_{*}^{d}}\Phi(\abs{x}^{-\alpha})\dd{\zeta(x)}<\infty.
 \end{equation}
It thus remains to prove \eqref{eq:gen:mom}. To do so, we take for $\eps\in(0,1)$ the regularisation $\Phi_\eps$ of $\Phi$ given by \eqref{eq:Phi:reg} in the weak formulation of \eqref{eq:mult:coag}. Note that $\Phi$ is smooth and $r\mapsto r^{-\alpha}$ is decaying such that $\Phi_\eps$ defines a suitable test function. Since $\Phi_\eps$ is subadditive according to Lemma~\ref{Lem:Phi:subadditive}, this then yields
\begin{equation*}
 \int_{\R_{*}^{d}}\Phi_{\eps}(x)\dd{f_{t}(x)}\leq \int_{\R_{*}^{d}}\Phi_{\eps}(x)\dd{f_{0}(x)}+t\int_{\R_{*}^{d}}\Phi_{\eps}(x)\dd{\zeta(x)} \qquad \text{for all  } t\geq 0.
\end{equation*}
Passing to the limit $\eps\to 0$, we thus deduce from \eqref{eq:small:moment:2:1} together with $\Phi_\eps(x)\to \Phi(\abs{x}^{-\alpha})$ that
\begin{equation*}
 \int_{\R_{*}^{d}}\Phi(\abs{x}^{-\alpha})\dd{f_{t}(x)}\leq \int_{\R_{*}^{d}}\Phi(\abs{x}^{-\alpha})\dd{f_{0}(x)}+t\int_{\R_{*}^{d}}\Phi(\abs{x}^{-\alpha})\dd{\zeta(x)}\qquad \text{for all  } t\geq 0.
\end{equation*}
\end{proof}

\begin{lemma}\label{Lem:Phi:subadditive}
 Let $\Phi\in C([0,\infty),[0,\infty))$ be monotonically increasing and let $\alpha>0$. For each $\eps\in(0,1)$ the function $\Phi_\eps\colon \R_{*}^{d}\to [0,\infty)$ given by
 \begin{equation}\label{eq:Phi:reg}
  \Phi_{\eps}(x)\vcc=\begin{cases}
                      \Phi(\eps^{-\alpha})\frac{\abs{x}}{\eps} & \text{if } \abs{x}<\eps\\
                      \Phi(\abs{x}^{-\alpha}) &\text{if } \abs{x}\geq \eps
                     \end{cases}
 \end{equation}
is subadditive and satisfies $\Phi_{\eps}(x)\leq \Phi(\abs{x}^{-\alpha})$ and $\Phi_\eps(x)\to \Phi(\abs{x}^{-\alpha})$ as $\eps\to0$.
\end{lemma}

\begin{proof}
 Let $x,y\in\R_{*}^{d}$. If $\abs{x+y}<\eps$ we have $\abs{x},\abs{y}<\eps$ and thus $\Phi_\eps(x+y)=\Phi_{\eps}(x)+\Phi_{\eps}(y)$ by linearity. On the other hand, if $\abs{x+y}\geq \eps$ we have either $\max\{\abs{x},\abs{y}\}\geq \eps$ or $\abs{x},\abs{y}<\eps$. In the first case, we conclude by monotonicity that
 \begin{equation*}
  \Phi_{\eps}(x+y)=\Phi(\abs{x+y}^{-\alpha})\leq \Phi(\max\{\abs{x},\abs{y}\}^{-\alpha})\leq \Phi_\eps(x)+\Phi_\eps(y).
 \end{equation*}
In the second case, i.e.\@ $\abs{x+y}\geq \eps$ and $\abs{x},\abs{y}<\eps$, we have similarly
\begin{equation*}
 \Phi_{\eps}(x+y)=\Phi(\abs{x+y}^{-\alpha}) \leq \Phi(\eps^{-\alpha})\leq \Phi(\eps^{-\alpha}) \frac{\abs{x}+\abs{y}}{\eps}=\Phi_\eps(x)+\Phi_\eps(y).
\end{equation*}
Finally, for $\abs{x}<\eps$ the monotonicity of $\Phi$ yields $\Phi_{\eps}(x)=\Phi(\eps^{-\alpha})\frac{\abs{x}}{\eps}\leq\Phi(\eps^{-\alpha})\leq \Phi(\abs{x}^{-\alpha})$.
\end{proof}

The following lemma provides an approximation of $x\mapsto \abs{x}^{k}$ for $k>1$ by suitable test functions which will be used to bound higher order moments.

\begin{lemma}\label{Lem:phi:reg}
Let $k>1$, $\eps\in(0,1)$ and $R>1$ and let $\varphi_{\eps,R}\colon \R_{*}^{d}\to [0,\infty)$ be given by
\begin{equation*}
  \varphi_{\eps,R}(x)\vcc=\begin{cases}
                           \eps^{k-1}\abs{x} &\text{if } \abs{x}<\eps\\
                           \min\{|x|,R\}^{k} & \text{if } \abs{x}\geq \eps.
                          \end{cases}
 \end{equation*}
There exists a constant $C_k>0$ such that we have for any $\mu\in[0,1]$ that
 \begin{equation*}
    \varphi_{\eps,R}(x+y)-\varphi_{\eps,R}(x)-\varphi_{\eps,R}(y)\leq C_{k}\begin{cases}
                                                                             |x|^{\mu}\min\{\abs{y},R\}^{k}|y|^{-\mu} & \text{if } |x|\leq |y|\\
                                                                             \min\{\abs{x},R\}^k|x|^{-\mu}|y|^{\mu}& \text{if } |y|\leq |x|
                                                                            \end{cases}
                                                                        \qquad \text{for all } x,y\in\R_{*}^{d}.
 \end{equation*}
 Moreover, $\varphi_{\eps,R}(x)\leq \omega_{-\theta_1,k}(x)$.
\end{lemma}

\begin{proof}
 We first recall from \cite{EsM06} that there exists $C_k>0$ such that 
  \begin{equation}\label{eq:phi:reg:1}
  |x+y|^{k}-|x|^{k}-|y|^{k}\leq C_{k}\begin{cases}
                                                                             |x|^{\mu}|y|^{k-\mu} & \text{if } |x|\leq |y|\\
                                                                             |x|^{k-\mu}|y|^{\mu}& \text{if } |y|\leq |x|
                                                                            \end{cases}
                                                                            \qquad 
                                                                            \text{for all } \mu\in[0,1].
 \end{equation}
 Moreover, if either $|x|\geq R$ or $|y|\geq R$ one immediately checks that $\varphi_{\eps,R}(x+y)-\varphi_{\eps,R}(x)-\varphi_{\eps,R}(y)\leq 0$. In the same way, if $\abs{x+y}<\eps$ we have $\varphi(x+y)-\varphi(x)-\varphi(y)\leq 0$. It thus suffices to restrict to $\abs{x+y}\geq \eps$ and $\abs{x},\abs{y}\leq R$ where we have
 \begin{equation*}
  \varphi_{\eps,R}(x+y)-\varphi_{\eps,R}(x)-\varphi_{\eps,R}(y)\leq \abs{x+y}^{k}-\varphi_{\eps,R}(x)-\varphi_{\eps,R}(y).
 \end{equation*}
Moreover, assuming $\abs{x}\leq \abs{y}$ (the case $\abs{y}\leq \abs{x}$ can be treated in the same way by symmetry) we obtain by means of \eqref{eq:phi:reg:1} that 
 \begin{equation*}
  \varphi_{\eps,R}(x+y)-\varphi_{\eps,R}(x)-\varphi_{\eps,R}(y)\leq \abs{x}^{k}+\abs{y}^{k}+C_{k}\abs{x}^{\mu}\abs{y}^{k-\mu}-\varphi_{\eps,R}(x)-\varphi_{\eps,R}(y).
 \end{equation*}
 This directly yields the first claim upon noting that $\abs{z}^{k}\leq \varphi_{\eps,R}(z)$ for $\abs{z}\leq R$ and $\abs{y}^{k}\leq \min\{\abs{y},R\}^{k}$ since $\abs{y}\leq R$.
 
 Clearly, we have $\min\{\abs{x},R\}^{k}\leq \abs{x}^{k}$ and $\eps^{k-1}\abs{x}\leq \abs{x}$ since $k>1$. Thus, if $\abs{x}\geq 1$ we immediately get $\varphi_{\eps,R}(x)\leq \omega_{-\theta_1,k}(x)$. On the other hand, if $\abs{x}\leq 1$, we note $\abs{x},\abs{x}^{k}\leq \abs{x}^{-\theta_1}$ due to $-\theta_1\leq 1$ which again gives $\varphi_{\eps,R}(x)\leq \omega_{-\theta_1,k}(x)$.
\end{proof}

The next lemma gives an estimate on superlinear moments for weak solutions provided that the initial value is bounded accordingly.

\begin{lemma}\label{Lem:higher:moments:1}
 Let $k>1$ and $f_0,\zeta\in\B_{-\theta_1,1}^{+}$ such that $\int_{\R_{*}^{d}}\abs{x}^{k}\dd{f_{0}}(x)<\infty$ and $\int_{\R_{*}^{d}}\abs{x}^{k}\dd{\zeta}(x)<\infty$. There exists a constant $\theta\in[0,1)$ and a continuous function $\psi\colon[0,\infty)\to [0,\infty)$ (which can be computed explicitly in terms of $f_0$, $\zeta$, $k$ and $\theta$) such that each weak solution $f_t$ of \eqref{eq:mult:coag:source} satisfies the estimate
 \begin{equation*}
  \int_{\R_{*}^{d}}\abs{x}^{k}\dd{f_t(x)}\leq \biggl(\biggl(\int_{\R_{*}^{d}}\abs{x}^{k}\dd{f_{0}}(x)+1\biggr)^{1-\theta}+\psi(t)\biggr)^{\frac{1}{1-\theta}}-1.
 \end{equation*}
\end{lemma}

\begin{proof}
 The proof is similar to \cite[Proof of Lemma~3.4.]{EsM06} but needs some adaption due to the larger class of kernels we are considering here. For $\eps\in(0,1)$ and $R>1$, we take the test function
 \begin{equation*}
  \varphi_{\eps,R}(x)\vcc=\begin{cases}
                           \eps^{k-1}\abs{x} &\text{if } \abs{x}<\eps\\
                           \min\{|x|,R\}^{k} & \text{if } \abs{x}\geq \eps
                          \end{cases}
 \end{equation*}
 in the definition of weak solution. Combining \eqref{eq:est:kernel:2} with the estimate on $\varphi_{\eps,R}$ from Lemma~\ref{Lem:phi:reg} we get
\begin{multline*}
 K(x,y)[\varphi_{\eps,R}(x+y)-\varphi_{\eps,R}(x)-\varphi_{\eps,R}(y)]\\*
 \leq C\bigl(|x|^{\theta_2-\mu}\min\{\abs{x},R\}^{k}|y|^{\mu-\theta_1}+|x|^{\mu-\theta_1}|y|^{\theta_2-\mu}\min\{\abs{y},R\}^{k}\bigr)
\end{multline*}
for each $\mu\in[0,1]$. From Lemma~\ref{Lem:small:moments:1} together with Remark~\ref{Rem:weak:sol} we deduce that
\begin{multline}\label{eq:high:moments:1}
   \frac{\dd}{\dd{t}}\int_{\R_{*}^{d}}\varphi_{\eps,R}(x)\dd{f_t(x)}\\*
   \leq C \int_{\R_{*}^{d}}\abs{x}^{\mu-\theta_1}\dd{f_{t}(x)}\int_{\R_{*}^{d}}|x|^{\theta_2-\mu}\min\{\abs{x},R\}^{k}\dd{f_{t}(x)}+\int_{\R_{*}^{d}}\varphi_{\eps,R}(x)\dd{\zeta(x)}\\*
   \leq C\bigl(\norm{f_{0}}_{-\theta_1,1}+t\norm{\zeta}_{-\theta_1,1}\bigr)\int_{\R_{*}^{d}}|x|^{\theta_2-\mu}\min\{\abs{x},R\}^{k}\dd{f_{t}(x)}+\int_{\R_{*}^{d}}\varphi_{\eps,R}(x)\dd{\zeta(x)}
\end{multline}
if $\mu\in[0,\min\{1,1+\theta_1\}]$. To proceed, we have to distinguish the two cases $k+\theta_2\leq 1$ and $k+\theta_2>1$: 

If $k+\theta_2\leq 1$, we choose $\mu=0$ and note that $\abs{x}^{\theta_2}\min\{\abs{x},R\}^{k}\leq \omega_{-\theta,1}(x)$ due to $-\theta_1\leq \theta_2$ and $k>1$ as well as $\varphi_{\eps,R}(x)\leq \omega_{0,k}(x)$ such that \eqref{eq:high:moments:1} together with Lemma~\ref{Lem:small:moments:1} yields
\begin{equation*}
  \frac{\dd}{\dd{t}}\int_{\R_{*}^{d}}\varphi_{\eps,R}(x)\dd{f_t(x)}\leq C \bigl(\norm{f_{0}}_{-\theta_1,1}+t\norm{\zeta}_{-\theta_1,1}\bigr)^2+\norm{\zeta}_{0,k}.
\end{equation*}
Integrating this inequality, we get
\begin{multline*}
 \int_{\R_{*}^{d}}\varphi_{\eps,R}(x)\dd{f_t(x)}\\*
 \leq \int_{\R_{*}^{d}}\varphi_{\eps,R}(x)\dd{f_0(x)}+ \bigl(C \norm{f_0}_{-\theta_1,1}^2+\norm{\zeta}_{0,k}\bigr) t+\norm{\zeta}_{-\theta_1,1}\bigl(\norm{f_0}_{-\theta_1,1}t^2+\frac{1}{3}\norm{\zeta}_{-\theta_1,1}t^3\bigr).
\end{multline*}
Passing to the limit $\eps\to0$ and $R\to\infty$ we have
\begin{equation*}
 \int_{\R_{*}^{d}}\abs{x}^{k}\dd{f_t(x)}\leq \int_{\R_{*}^{d}}\abs{x}^{k}\dd{f_0(x)}+ \bigl(C \norm{f_0}_{-\theta_1,1}^2+\norm{\zeta}_{0,k}\bigr) t+\norm{\zeta}_{-\theta_1,1}\bigl(\norm{f_0}_{-\theta_1,1}t^2+\frac{1}{3}\norm{\zeta}_{-\theta_1,1}t^3\bigr).
\end{equation*}
This proves the claim in the case $k+\theta_2\leq 1$. On the other hand, if $k+\theta_2>1$, we choose
\begin{equation*}
 \mu=\theta_2+\min\Bigl\{1-\theta_2,1+\theta_1-\theta_2,\frac{\max\{0,-\theta_2\}+k-1}{2}\Bigr\}=\min\Bigl\{1,1+\theta_1,\frac{\max\{\theta_2,0\}+k-1}{2}\Bigr\}.
\end{equation*}
By means of Hölder's inequality with $\theta=1-\frac{\mu-\theta_2}{k-1}$ we get together with $\min\{\abs{x},R\}\leq \abs{x}$ and the (generalised) conservation of mass, i.e.\@ $\norm{f_t}_{1,1}=\norm{f_0}_{1,1}+t\norm{\zeta}_{1,1}$ for all $t\geq 0$, that
\begin{multline*}
 \int_{\R_{*}^{d}}\min\{|x|,R\}^{k+\theta_2-\mu}\dd{f_{t}(x)}\\*
 \leq \biggl(\int_{\R_{*}^{d}}\min\{|x|,R\}^{(k+\theta_2-\mu-(1-\theta))/\theta}\dd{f_{t}(x)}\biggr)^{\theta}\biggl(\int_{\R_{*}^{d}}\min\{|x|,R\}\dd{f_{t}(x)}\biggr)^{1-\theta}\\*
 \leq \biggl(\int_{\R_{*}^{d}}\min\{\abs{x},R\}^{k}\dd{f_{t}(x)}\biggr)^{\theta}\bigl(\norm{f_0}_{1,1}+t\norm{\zeta}_{1,1}\bigr)^{1-\theta}.
\end{multline*}
We plug this into \eqref{eq:high:moments:1} noting also that $|x|^{\theta_2-\mu}\min\{\abs{x},R\}^{k}\leq \min\{|x|,R\}^{k+\theta_2-\mu}$ since $\mu\geq \theta_2$ and $\min\{\abs{x},R\}^{k}\leq \varphi_{\eps,R}(x)$ since $k>1$ as well as $(\norm{f_0}_{1,1}+t\norm{\zeta}_{1,1})\leq (\norm{f_0}_{-\theta_1,1}+t\norm{\zeta}_{-\theta_1,1})$ which yields
\begin{equation*}
 \frac{\dd}{\dd{t}}\int_{\R_{*}^{d}}\varphi_{\eps,R}(x)\dd{f_{t}(x)}\leq C(\norm{f_0}_{-\theta_1,1}+t\norm{\zeta}_{-\theta_1,1})^{2-\theta}\biggl(\int_{\R_{*}^{d}}\varphi_{\eps,R}(x)\dd{f_{t}(x)}\biggr)^{\theta}+\int_{\R_{*}^{d}}\varphi_{\eps,R}(x)\dd{\zeta(x)}.
\end{equation*}
We estimate the right-hand side further to get
\begin{multline*}
 \frac{\dd}{\dd{t}}\int_{\R_{*}^{d}}\varphi_{\eps,R}(x)\dd{f_{t}(x)}\\*
 \leq \biggl(C(\norm{f_0}_{-\theta_1,1}+t\norm{\zeta}_{-\theta_1,1})^{2-\theta}+\int_{\R_{*}^{d}}\varphi_{\eps,R}(x)\dd{\zeta(x)}\biggr)\biggl(\int_{\R_{*}^{d}}\varphi_{\eps,R}(x)\dd{f_{t}(x)}+1\biggr)^{\theta}.
\end{multline*}
Integrating the previous inequality, we obtain
\begin{multline*}
 \int_{\R_{*}^{d}}\varphi_{\eps,R}(x)\dd{f_{t}(x)}\leq \Biggl(\biggl(\int_{\R_{*}^{d}}\varphi_{\eps,R}(x)\dd{f_{0}(x)}+1\biggr)^{1-\theta}\\*
 +C\int_{0}^{t}(\norm{f_0}_{-\theta_1,1}+s\norm{\zeta}_{-\theta_1,1})^{2-\theta}\dd{s}+t\int_{\R_{*}^{d}}\varphi_{\eps,R}(x)\dd{\zeta(x)}\Biggr)^{\frac{1}{1-\theta}}-1.
\end{multline*}
Passing to the limit $\eps\to 0$ and $R\to\infty$ finishes the proof.
\end{proof}

\section{Uniqueness}\label{Sec:uniqueness}

In this section, we will give the proof of Theorem~\ref{Thm:unique}. The approach is similar to the one for the Boltzmann equation in \cite{LuM12} and exploits that weak solutions are actually strong solutions. However some adaptations are required here since the coagulation model \eqref{eq:mult:coag:source} does not provide conservation of moments which was exploited in \cite{LuM12}. Instead, we rely on the monotonicity property for sub-linear moments as proven in Lemmas \ref{Lem:small:moments:1} and \ref{Lem:small:moments:2}.

As a preparatory step, we provide the following estimate on the strong form of the collision operator given in \eqref{eq:def:coag:op:strong} which is thus well-defined for measures having sufficiently nice moment bounds. Moreover, it  provides estimates on differences which will be exploited in Proposition~\ref{Prop:weak:strong} below.

\begin{lemma}\label{Lem:cont:coag:op:gen}
  Let $K$ satisfy \eqref{eq:est:kernel:2}. Then $Q^{\pm}\colon \B_{-\theta_1,\beta+\theta_2}(\R_{*}^{d})\times\B_{-\theta_1,\beta+\theta_2}(\R_{*}^{d}) \to \B_{\alpha,\beta}(\R_{*}^{d})$ given by \eqref{eq:def:coag:op:strong} is well-defined for each $\alpha,\beta\geq 0$, i.e.\@ $\norm{Q^{\pm}(\mu,\nu)}_{\alpha,\beta}\lesssim \norm{\mu}_{-\theta_1,\beta+\theta_2}\norm{\nu}_{-\theta_1,\beta+\theta_2}$. Moreover, we have
 \begin{equation*}
   \norm{Q^{\pm}(\mu,\mu)-Q^{\pm}(\nu,\nu)}_{\alpha,\beta}\lesssim \norm{\mu-\nu}_{-\theta_1,\beta+\theta_2}\norm{\mu+\nu}_{-\theta_1,\beta+\theta_2}
 \end{equation*}
 for each $\mu,\nu\in\B_{-\theta_1,\beta+\theta_2}$.
\end{lemma}

\begin{proof}
 The estimate \eqref{eq:est:kernel:2} directly implies $ K(x,y)\leq c_{u}\bigl(\abs{x}^{-\theta_1}\abs{y}^{\theta_2}+\abs{x}^{\theta_2}\abs{y}^{-\theta_1}\bigr)$. Let $\varphi\in C_{c}(\R_{*}^{d})\cap L_{-\alpha,-\beta}^{\infty}(\R_{*}^{d})$ with $\norm{\varphi}_{L_{-\alpha,-\beta}^{\infty}}\leq 1$. From \eqref{eq:def:coag:op:strong} we get
 \begin{multline*}
  \abs*{\int_{\R_{*}^{d}}\varphi(x)\dd{Q^{+}(\mu,\nu)(x)}}=\abs*{\int_{(\R_{*}^{d})^2}\varphi(x+y)K(x,y)\dd{\mu(x)}\dd{\nu(y)}}\\* \leq  c_u \int_{(\R_{*}^{d})^{2}}\omega_{\alpha,\beta}(x+y)\bigl(\abs{x}^{-\theta_1}\abs{y}^{\theta_2}+\abs{x}^{\theta_2}\abs{y}^{-\theta_1}\bigr)\dd{\abs{\mu}(x)}\dd{\abs{\nu}(y)}.
 \end{multline*}
 Together with Lemma~\ref{Lem:weight:1} we deduce
  \begin{multline*}
  \abs*{\int_{\R_{*}^{d}}\varphi(x)\dd{Q^{+}(\mu,\nu)(x)}}\\*
  \shoveleft{\leq  \max\{2^{\alpha},2^{\beta}\}c_u \int_{(\R_{*}^{d})^{2}}\omega_{\alpha,\beta}(x)\bigl(\abs{x}^{-\theta_1}\abs{y}^{\theta_2}+\abs{x}^{\theta_2}\abs{y}^{-\theta_1}\bigr)\dd{\abs{\mu}(x)}\dd{\abs{\nu}(y)}}\\*
  \shoveright{+ \max\{2^{\alpha},2^{\beta}\}c_u \int_{(\R_{*}^{d})^{2}}\omega_{\alpha,\beta}(y)\bigl(\abs{x}^{-\theta_1}\abs{y}^{\theta_2}+\abs{x}^{\theta_2}\abs{y}^{-\theta_1}\bigr)\dd{\abs{\mu}(x)}\dd{\abs{\nu}(y)}}\\
  \shoveleft{\leq \max\{2^{\alpha},2^{\beta}\}c_u \bigl(\norm{\mu}_{\alpha-\theta_1,\beta-\theta_1}\norm{\nu}_{\theta_2,\theta_2}}+\norm{\mu}_{\alpha+\theta_2,\beta+\theta_2}\norm{\nu}_{-\theta_1,-\theta_1}\\*
  \shoveright{+\norm{\mu}_{-\theta_1,-\theta_1}\norm{\nu}_{\alpha+\theta_2,\beta+\theta_2}+\norm{\mu}_{\theta_2,\theta_2}\norm{\nu}_{\alpha-\theta_1,\beta-\theta_1}\bigr)}\\*
  \leq 4\max\{2^{\alpha},2^{\beta}\}c_u \norm{\mu}_{-\theta_1,\beta+\theta_2}\norm{\nu}_{-\theta_1,\beta+\theta_2}.
 \end{multline*}
 Recalling~\eqref{eq:duality}, the first estimate on $Q^{+}$ follows. In the same way we get
\begin{multline*}
 \abs*{\int_{\R_{*}^{d}}\varphi(x)\dd{Q^{-}(\mu,\nu)(x)}}=\abs*{\int_{(\R_{*}^{d})^2}\varphi(x)K(x,y)\dd{\mu(x)}\dd{\nu(y)}}\\*
 \shoveleft{\leq c_u \int_{(\R_{*}^{d})^{2}}\omega_{\alpha,\beta}(x)\bigl(\abs{x}^{-\theta_1}\abs{y}^{\theta_2}+\abs{x}^{\theta_2}\abs{y}^{-\theta_1}\bigr)\dd{\abs{\mu}(x)}\dd{\abs{\nu}(y)}}\\*
 \leq c_u \bigl(\norm{\mu}_{\alpha-\theta_1,\beta-\theta_1}\norm{\nu}_{\theta_2,\theta_2}+\norm{\mu}_{\alpha+\theta_2,\beta+\theta_2}\norm{\nu}_{-\theta_1,-\theta_1}\bigr)\leq c_u \norm{\mu}_{-\theta_1,\beta+\theta_2}\norm{\nu}_{-\theta_1,\beta+\theta_2}.
\end{multline*}
The last estimate is a direct consequence of the bi-linearity of $Q^{\pm}$.
\end{proof}

As announced above, we can now show the key ingredient for the proof of uniqueness, namely that weak solutions are actually strong solutions.

\begin{proposition}\label{Prop:weak:strong}
 Let $K$ satisfy \eqref{eq:est:kernel:2}. Then every weak solution of \eqref{eq:mult:coag:source} is a strong solution.
\end{proposition}

\begin{proof}
 Let $f_t$ be a weak solution according to Definition~\ref{Def:weak:sol}. From Lemma~\ref{Lem:small:moments:1} and $\omega_{-\theta_1,\theta_2}(x)\leq \omega_{-\theta_1,1}(x)$ we get for all $t\geq 0$ that
 \begin{equation}\label{eq:strong:cont:1}
   \int_{\R_{*}^{d}}\omega_{-\theta_1,\theta_2}(x)\dd{f_{t}(x)}\leq \int_{\R_{*}^{d}}\omega_{-\theta_1,\theta_2}(x)\dd{f_{0}(x)}+t\int_{\R_{*}^{d}}\omega_{-\theta_1,\theta_2}(x)\dd{\zeta(x)}\leq \norm{f_0}_{-\theta_1,1}+t\norm{\zeta}_{-\theta_1,1},
 \end{equation}
i.e.\@ $\norm{f_t}_{-\theta_1,\theta_2}\leq \norm{f_0}_{-\theta_1,1}+t\norm{\zeta}_{-\theta_1,1}$. Lemma \ref{Lem:cont:coag:op:gen} thus yields $\norm{Q^{\pm}(f_t,f_t)}_{\max\{0,-\theta_1\},0}\lesssim  (\norm{f_0}_{-\theta_1,1}+t\norm{\zeta}_{-\theta_1,1})^{2}$. In particular, the weak and the strong definition of $Q$ coincide for $f_t$ and we get from Definition~\ref{Def:weak:sol} for $\varphi\in\mathcal{T}$ with $\norm{\varphi}_{L^{\infty}_{-\max\{0,-\theta_1\},0}}\leq 1$ that
 \begin{equation*}
  \int_{\R_{*}^{d}}\varphi(x)\dd{f_t(x)}=\int_{\R_{*}^{d}}\varphi(x)\dd{f_0(x)}+\int_{0}^{t}\int_{\R_{*}^{d}} \varphi(x)\dd{Q(f_s,f_s)(x)}\dd{s}+t\int_{\R_{*}^{d}}\varphi(x)\dd{\zeta(x)}.
 \end{equation*}
 Thus, for $0\leq t_1\leq t_2<\infty$ we obtain the estimate
 \begin{multline*}
  \abs[\bigg]{\int_{\R_{*}^{d}}\varphi(x)\dd{(f_{t_2}-f_{t_1})(x)}}\leq \int_{t_1}^{t_2}\abs[\Big]{\int_{\R_{*}^{d}} \varphi(x)\dd{Q(f_s,f_s)(x)}}\dd{s}+(t_2-t_1)\int_{\R_{*}^{d}}\varphi(x)\dd{\zeta(x)}\\*
  \leq\int_{t_1}^{t_2}\norm{Q(f_s,f_s)}_{\max\{0,-\theta_1\},0}\dd{s}+(t_2-t_1)\norm{\zeta}_{\max\{0,-\theta_1\},0}\\*
  \leq C \int_{t_1}^{t_2}(\norm{f_0}_{-\theta_1,1}+s\norm{\zeta}_{-\theta_1,1})^{2}\dd{s}+(t_2-t_1)\norm{\zeta}_{\max\{0,-\theta_1\},0}\\*
  \leq \Bigl(C\bigl(\norm{f_0}_{-\theta_1,1}+(t_1+t_2)\norm{\zeta}_{-\theta_1,1}\bigr)^2+\norm{\zeta}_{\max\{0,-\theta_1\},0}\Bigr)\abs{t_2-t_1}.
 \end{multline*}
 By duality, we deduce 
 \begin{equation}\label{eq:weak:strong:2}
  \norm{f_{t_2}-f_{t_1}}_{\max\{0,-\theta_1\},0}\leq \Bigl(C\bigl(\norm{f_0}_{-\theta_1,1}+(t_1+t_2)\norm{\zeta}_{-\theta_1,1}\bigr)^2+\norm{\zeta}_{\max\{0,-\theta_1\},0}\Bigr)\abs{t_2-t_1}.
 \end{equation}
 To prove the strong continuity of $f_t$ in $\B_{-\theta_1,1}$, we first fix $t_0\in [0,\infty)$ and note that $\norm{f_{t}-f_{t_0}}_{-\theta_1,1}\leq \norm{f_{t}-f_{t_0}}_{1,1}+\norm{f_{t}-f_{t_0}}_{-\theta_1,-\theta_1}$. To proceed, we have to distinguish whether $\theta_1>0$ or $\theta_1\leq 0$. In the first case, i.e.\@ $\theta_1>0$, taking $0<r<1<R<\infty$ we split the integrals and use $\abs{f_t-f_{t_0}}=(f_t-f_{t_{0}})+2(f_{t_{0}}-f_{t})^{+}$ together with the mass conservation to get
 \begin{multline*}
  \norm{f_{t}-f_{t_0}}_{-\theta_1,1}\leq \abs{t-t_0}\int_{\R_{*}^{d}}\abs{x}\dd{\zeta(x)}+2\int_{\abs{x}\leq R}\abs{x}\dd{(f_{t_{0}}-f_{t})^{+}(x)}+2\int_{\abs{x}> R}\abs{x}\dd{(f_{t_{0}}-f_{t})^{+}(x)}\\*
  + \int_{\abs{x}< r}\abs{x}^{-\theta_1}\dd{\abs{f_{t}-f_{t_0}}(x)}+\int_{r\leq \abs{x}\leq R}\abs{x}^{-\theta_1}\dd{\abs{f_{t}-f_{t_0}}(x)}+\int_{\abs{x}>R}\abs{x}^{-\theta_1}\dd{\abs{f_{t}-f_{t_0}}(x)}.
 \end{multline*}
 Since $(f_{t_{0}}-f_{t})^{+}\leq f_{t_{0}}$, we deduce together with \eqref{eq:weak:strong:2}, Lemma~\ref{Lem:small:moments:2} and the conservation of mass that
  \begin{multline}\label{eq:weak:strong:3}
  \norm{f_{t}-f_{t_0}}_{-\theta_1,1}\leq \Bigl((2R+r^{-\theta_1})\bigl(C(\norm{f_0}_{-\theta_1,1}+(t_0+t)\norm{\zeta}_{-\theta_1,1})^2+\norm{\zeta}_{\max\{0,-\theta_1\},0}\bigr)+\norm{\zeta}_{1,1}\Bigr)\abs{t-t_0}\\*
  +2\int_{\abs{x}> R}\abs{x}\dd{f_{t_{0}}(x)}+ \frac{r^{-\theta_1}}{\Phi(r^{-\theta_1})}\int_{\abs{x}< r}\Phi(\abs{x}^{-\theta_1})\dd{\abs{f_{t}-f_{t_0}}(x)}+R^{-(1+\theta_1)}\int_{\abs{x}>R}\abs{x}\dd{\abs{f_{t}-f_{t_0}}(x)}\\*
   \leq \Bigl((2R+r^{-\theta_1})\bigl(C(\norm{f_0}_{-\theta_1,1}+(t_0+t)\norm{\zeta}_{-\theta_1,1})^2+\norm{\zeta}_{\max\{0,-\theta_1\},0}\bigr)+\norm{\zeta}_{1,1}\Bigr)\abs{t-t_0}\\*
   +2\int_{\abs{x}> R}\abs{x}\dd{f_{t_{0}}(x)} + \frac{r^{-\theta_1}}{\Phi(r^{-\theta_1})}\biggl(2\int_{\R_{*}^{d}}\Phi(\abs{x}^{-\theta_1})\dd{f_{0}(x)}+(t+t_0)\int_{\R_{*}^{d}}\Phi(\abs{x}^{-\theta_1})\dd{\zeta(x)}\biggr)\\*
   +R^{-(1+\theta_1)}\biggl(2\int_{\R_{*}^{d}}\abs{x}\dd{f_{0}(x)}+(t+t_0)´\int_{\R_{*}^{d}}\abs{x}\dd{\zeta(x)}\biggr).
 \end{multline}
 If $\theta_1\leq 0$, we can basically proceed in the same way but there is no need to consider the region $\abs{x}<r$ separately. More precisely, we use the splitting $\norm{f_{t}-f_{t_0}}_{-\theta_1,-\theta_1}=\int_{\abs{x}\leq R}\abs{x}^{-\theta_1}\dd{\abs{f_{t}-f_{t_0}}(x)}+\int_{\abs{x}>R}\abs{x}^{-\theta_1}\dd{\abs{f_{t}-f_{t_0}}(x)}$ and note that~\eqref{eq:weak:strong:2} implies 
 \begin{equation*}
  \int_{\abs{x}\leq R}\abs{x}^{-\theta_1}\dd{\abs{f_{t}-f_{t_0}}(x)}\leq (1+R^{-\theta_1})\Bigl(C\bigl(\norm{f_0}_{-\theta_1,1}+(t_0+t)\norm{\zeta}_{-\theta_1,1}\bigr)^2+\norm{\zeta}_{\max\{0,-\theta_1\},0}\Bigr)\abs{t-t_0}.
 \end{equation*}
Thus, we get in the same way as in \eqref{eq:weak:strong:3} that
  \begin{multline}\label{eq:weak:strong:4}
  \norm{f_{t}-f_{t_0}}_{-\theta_1,1}\\*
  \leq \Bigl((2R+1+R^{-\theta_1})\bigl(C(\norm{f_0}_{-\theta_1,1}+(t_0+t)\norm{\zeta}_{-\theta_1,1})^2+\norm{\zeta}_{\max\{0,-\theta_1\},0}\bigr)+\norm{\zeta}_{1,1}\Bigr)\abs{t-t_0} \\*
  +2\int_{\abs{x}> R}\abs{x}\dd{f_{t_{0}}(x)}+R^{-(1+\theta_1)}\biggl(2\int_{\R_{*}^{d}}\abs{x}\dd{f_{0}(x)}+(t+t_0)\int_{\R_{*}^{d}}\abs{x}\dd{\zeta(x)}\biggr).
 \end{multline}
Passing to the limit $t\to t_0$ in \eqref{eq:weak:strong:3} or \eqref{eq:weak:strong:4} respectively and then $R\to\infty$ if $\theta_1\leq 0$ or $r\to0$ and $R\to\infty$ if $\theta_1>0$, yields the strong continuity of $f_t$ in $\B_{-\theta_1,1}$.

Moreover, thanks to Lemma~\ref{Lem:cont:coag:op:gen} we get the strong continuity of $t\mapsto Q^{\pm}(f_t,f_t)$ in $\B_{\max\{0,-\theta_1\},0}$. From here we can proceed in the same way as in \cite[Proof of Theorem~1.5, part (a)]{LuM12} to prove that $f_t$ is a strong solution to \eqref{eq:mult:coag:source}.
\end{proof}

We are now prepared to give the proof of the main statement, i.e.\@ that weak solutions as given in Definition~\ref{Def:weak:sol} are unique.

\begin{proof}[Proof of Theorem~\ref{Thm:unique}]
 Let $f_t$ and $g_t$ be weak solutions to~\eqref{eq:mult:coag:source} with the same initial condition $f_0$. Let $\sigma_t=\sgn(f_t-g_t)$, i.e.\@ the Borel function $\sigma_t\colon \R_{*}^{d}\to \R$ such that $\abs{\sigma_t}=1$ and $\dd{\abs{f_t-g_t}}=\sigma_t\dd{(f_t-g_t)}$. According to Proposition~\ref{Prop:weak:strong}, $f_t$ and $g_t$ are strong solutions and we have
 \begin{equation*}
  \begin{split}
   f_t&=f_0+\int_{0}^{t}Q(f_s,f_s)\dd{s}+\zeta\\
   g_t&=f_0+\int_{0}^{t}Q(g_s,g_s)\dd{s}+\zeta.
  \end{split}
 \end{equation*}
Taking the difference, this yields
\begin{equation*}
 f_t-g_t=\int_{0}^{t}\bigl(Q(f_s,f_s)-Q(g_s,g_s)\bigr)\dd{s}.
\end{equation*}
Definition~\ref{Def:strong:sol} together with Lemma~\ref{Lem:cont:coag:op:gen} implies $Q(f_\cdot,f_\cdot)-Q(g_\cdot,g_\cdot)\in C([0,\infty),\B_{\max\{0,-\theta_1\},0}(\R_{*}^{d}))$. For each $n>1$ fixed, we define $\varphi_{n}(x)\vcc=\min\{n,\omega_{-\theta_1,\theta_2}(x)\}$ which is bounded. Thus, together with \cite[Lemma~5.1]{LuM12} we get
\begin{multline*}
 \int_{\R_{*}^{d}}\varphi_{n}(x)\dd{\abs{f_t-g_t}(x)}\\*
 \shoveleft{=\int_{\R_{*}^{d}}\varphi_{n}(x)\sigma_{t}(x)\dd{(f_t-g_t)(x)}=\int_{0}^{t}\int_{\R_{*}^{d}}\sigma_{s}(x)\varphi_{n}(x)\dd{(Q(f_s,f_s)-Q(g_s,g_s))(x)}\dd{s}}\\*
 =\frac{1}{2}\int_{0}^{t}\int_{(\R_{*}^d)^2}\bigl[\sigma_{s}(x+y)\varphi_{n}(x+y)-\sigma_{s}(x)\varphi_{n}(x)-\sigma_{s}(y)\varphi_{n}(y)\bigr]K(x,y)\dd{(f_s\otimes f_s-g_s\otimes g_s)}\dd{s}\\*
 =\frac{1}{2}\int_{0}^{t}\int_{(\R_{*}^d)^2}\bigl[\sigma_{s}(x+y)\varphi_{n}(x+y)-\sigma_{s}(x)\varphi_{n}(x)-\sigma_{s}(y)\varphi_{n}(y)\bigr]K(x,y)\dd{f_s(x)}\dd{(f_s-g_s)(y)}\dd{s}\\*
 +\frac{1}{2}\int_{0}^{t}\int_{(\R_{*}^d)^2}\bigl[\sigma_{s}(x+y)\varphi_{n}(x+y)-\sigma_{s}(x)\varphi_{n}(x)-\sigma_{s}(y)\varphi_{n}(y)\bigr]K(x,y)\dd{(f_s-g_s)(x)}\dd{g_s(y)}\dd{s}.
\end{multline*}
Using that $\sigma_{s}\dd{(f_s-g_s)}=\dd{\abs{f_s-g_s}}$ together with the symmetry of $K$, we deduce
\begin{multline*}
 \int_{\R_{*}^{d}}\varphi_{n}(x)\dd{\abs{f_t-g_t}(x)}\\*
 =\frac{1}{2}\int_{0}^{t}\int_{(\R_{*}^d)^2}\bigl[\sigma_{s}(x+y)\varphi_{n}(x+y)-\sigma_{s}(x)\varphi_{n}(x)\bigr]K(x,y)\dd{(f_s+g_s)(x)}\dd{(f_s-g_s)(y)}\dd{s}\\*
 -\frac{1}{2}\int_{0}^{t}\int_{(\R_{*}^d)^2}\varphi_{n}(y)K(x,y)\dd{(f_s+g_s)(x)}\dd{\abs{f_s-g_s}(y)}\dd{s}.
\end{multline*}
Since $\abs{\sigma_s \varphi_n}\leq \varphi_n$ and $\dd{(f_s-g_s)}\leq \dd{\abs{f_s-g_s}}$ we get the estimate
\begin{multline*}
 \int_{\R_{*}^{d}}\varphi_{n}(x)\dd{\abs{f_t-g_t}(x)}\\*
  \leq 
 \frac{1}{2}\int_{0}^{t}\int_{(\R_{*}^d)^2}\bigl[\varphi_{n}(x+y)+\varphi_{n}(x)-\varphi_{n}(y)\bigr]K(x,y)\dd{(f_s+g_s)(x)}\dd{\abs{f_s-g_s}(y)}\dd{s}.
\end{multline*}
Due to Lemma~\ref{Lem:weight:subadditive} we have $\varphi_{n}(x+y)-\varphi_n(y)\leq \varphi_{n}(x)$ which yields
\begin{equation*}
 \int_{\R_{*}^{d}}\varphi_{n}(x)\dd{\abs{f_t-g_t}(x)} \leq 
 \int_{0}^{t}\int_{(\R_{*}^d)^2}\varphi_{n}(x)K(x,y)\dd{(f_s+g_s)(x)}\dd{\abs{f_s-g_s}(y)}\dd{s}.
\end{equation*}
Moreover, $\varphi_{n}(x)\leq \omega_{-\theta_1,\theta_2}(x)$ and we get
\begin{equation*}
 \int_{\R_{*}^{d}}\varphi_{n}(x)\dd{\abs{f_t-g_t}(x)} \leq 
 \int_{0}^{t}\int_{(\R_{*}^d)^2}\omega_{-\theta_1,\theta_2}(x)K(x,y)\dd{(f_s+g_s)(x)}\dd{\abs{f_s-g_s}(y)}\dd{s}.
\end{equation*}
Since \eqref{eq:est:kernel:2} yields $K(x,y)\leq c_{u}(\abs{x}^{-\theta_1}\abs{y}^{\theta_2}+\abs{x}^{\theta_2}\abs{y}^{-\theta_1})$ we obtain together with $\theta_2-\theta_1=\gamma$ that
\begin{multline*}
  \int_{\R_{*}^{d}}\varphi_{n}(x)\dd{\abs{f_t-g_t}(x)}\\*
   \leq 
 c_{u}\int_{0}^{t}\int_{(\R_{*}^d)^2}\omega_{-\theta_1,\theta_2}(x)\bigl(\abs{x}^{-\theta_1}\abs{y}^{\theta_2}+\abs{x}^{\theta_2}\abs{y}^{-\theta_1}\bigr)\dd{(f_s+g_s)(x)}\dd{\abs{f_s-g_s}(y)}\dd{s}\\*
 \leq  c_{u}\int_{0}^{t}\int_{(\R_{*}^d)^2}\bigl(\omega_{-2\theta_1,\theta_2-\theta_1}(x)\abs{y}^{\theta_2}+\omega_{-\theta_1+\theta_2,2\theta_2}(x)\abs{y}^{-\theta_1}\bigr)\dd{(f_s+g_s)(x)}\dd{\abs{f_s-g_s}(y)}\dd{s}\\*
 =c_{u}\int_{0}^{t}\int_{(\R_{*}^d)^2}\bigl(\omega_{-2\theta_1,\gamma}(x)\abs{y}^{\theta_2}+\omega_{\gamma,2\theta_2}(x)\abs{y}^{-\theta_1}\bigr)\dd{(f_s+g_s)(x)}\dd{\abs{f_s-g_s}(y)}\dd{s}.
\end{multline*}
We have $-2\theta_1\leq \gamma\leq 2\theta_2$ and $-\theta_1\leq \theta_2$ which further implies
\begin{equation}\label{eq:unique:1}
 \int_{\R_{*}^{d}}\varphi_{n}(x)\dd{\abs{f_t-g_t}(x)} \leq 2c_{u}\int_{0}^{t}\int_{(\R_{*}^d)^2}\omega_{-2\theta_1,2\theta_2}(x)\omega_{-\theta_1,\theta_2}(y)\dd{(f_s+g_s)(x)}\dd{\abs{f_s-g_s}(y)}\dd{s}.
\end{equation}
If $2\theta_2\leq 1$, Lemma~\ref{Lem:small:moments:1} gives
\begin{equation*}
 \int_{\R_{*}^{d}}\varphi_{n}(x)\dd{\abs{f_t-g_t}(x)} \leq 4c_{u}\int_{0}^{t}\bigl(\norm{f_0}_{-2\theta_1,2\theta_2}+s\norm{\zeta}_{-2\theta_1,2\theta_2}\bigr)\int_{\R_{*}^d}\omega_{-\theta_1,\theta_2}(y)\dd{\abs{f_s-g_s}(y)}\dd{s}.
\end{equation*}
Grönwall's inequality then yields the claim upon passing to the limit $n\to\infty$.

If $2\theta_2> 1$, according to Lemmas \ref{Lem:small:moments:1} and \ref{Lem:higher:moments:1} there exist $\theta\in[0,1)$ and a continuous function $\psi\colon[0,\infty)\to [0,\infty)$ such that 
\begin{multline*}
    \int_{\R_{*}^d}\omega_{-2\theta_1,2\theta_2}(x)\dd{(f_s+g_s)(x)}\leq 2 \biggl(\int_{\R_{*}^{d}}\abs{x}^{-2\theta_1}\dd{f_{0}(x)}+s \int_{\R_{*}^{d}}\abs{x}^{-2\theta_1}\dd{\zeta(x)}\biggr)\\*
    + 2  \Biggl(\biggl(\biggl(\int_{\R_{*}^{d}}\abs{x}^{2\theta_2}\dd{f_{0}}(x)+1\biggr)^{1-\theta}+\psi(s)\biggr)^{\frac{1}{1-\theta}}-1\Biggr)=\vcc \Lambda(s).
\end{multline*}
Thus, from \eqref{eq:unique:1} we deduce
\begin{equation*}
 \int_{\R_{*}^{d}}\varphi_{n}(x)\dd{\abs{f_t-g_t}(x)} \leq 2c_{u}\int_{0}^{t}\Lambda(s)\int_{\R_{*}^d}\omega_{-\theta_1,\theta_2}(y)\dd{\abs{f_s-g_s}(y)}\dd{s}
\end{equation*}
and we conclude again by passing to the limit $n\to\infty$ and applying Grönwall's inequality.
\end{proof}

\end{document}